\documentclass[11pt]{amsart}
\usepackage{amscd}
\usepackage{amsmath}
\usepackage{amsxtra}
\usepackage{amsfonts}
\usepackage{amssymb}
\usepackage{color}

\newtheorem{theorem}{Theorem}[section]

\newtheorem{lemma}[theorem]{Lemma}
\newtheorem{proposition}[theorem]{Proposition}

\newtheorem{conjecture}[theorem]{Conjecture}
\theoremstyle{definition}
\newtheorem{definition}[theorem]{Definition}
\newtheorem{remark}[theorem]{Remark}

\newtheorem{example}[theorem]{Example}
\theoremstyle{remark}

\renewcommand{\theclaim}{\textup{\theclaim}}

\newtheorem*{acknowledgements}{Acknowledgements}

\numberwithin{equation}{section}

\def\openone

{\mathchoice

{\hbox{\upshape \small1\kern-3.3pt\normalsize1}}

{\hbox{\upshape \small1\kern-3.3pt\normalsize1}}

{\hbox{\upshape \tiny1\kern-2.3pt\SMALL1}}

{\hbox{\upshape \Tiny1\kern-2pt\tiny1}}}

\makeatletter

\newbox\ipbox

\newcommand{\ip}[2]{\left\langle #1\, , \,#2\right\rangle}
\newcommand{\diracb}[1]{\left\langle #1\mathrel{\mathchoice

{\setbox\ipbox=\hbox{$\displaystyle \left\langle\mathstrut
#1\right.$}

\vrule height\ht\ipbox width0.25pt depth\dp\ipbox}

{\setbox\ipbox=\hbox{$\textstyle \left\langle\mathstrut
#1\right.$}

\vrule height\ht\ipbox width0.25pt depth\dp\ipbox}

{\setbox\ipbox=\hbox{$\scriptstyle \left\langle\mathstrut
#1\right.$}

\vrule height\ht\ipbox width0.25pt depth\dp\ipbox}

{\setbox\ipbox=\hbox{$\scriptscriptstyle \left\langle\mathstrut
#1\right.$}

\vrule height\ht\ipbox width0.25pt depth\dp\ipbox}

}\right. }

\newcommand{\dirack}[1]{\left. \mathrel{\mathchoice

{\setbox\ipbox=\hbox{$\displaystyle \left.\mathstrut
#1\right\rangle$}

\vrule height\ht\ipbox width0.25pt depth\dp\ipbox}

{\setbox\ipbox=\hbox{$\textstyle \left.\mathstrut
#1\right\rangle$}

\vrule height\ht\ipbox width0.25pt depth\dp\ipbox}

{\setbox\ipbox=\hbox{$\scriptstyle \left.\mathstrut
#1\right\rangle$}

\vrule height\ht\ipbox width0.25pt depth\dp\ipbox}

{\setbox\ipbox=\hbox{$\scriptscriptstyle \left.\mathstrut
#1\right\rangle$}

\vrule height\ht\ipbox width0.25pt depth\dp\ipbox}

} #1\right\rangle}

\newcommand{\bz}{\mathbb{Z}}
\newcommand{\M}{\mathcal{M}}

\newcommand{\br}{\mathbb{R}}

\newcommand{\bn}{\mathbb{N}}

\newcommand{\beq}{\begin{equation}}
\newcommand{\eeq}{\end{equation}}

\def\blfootnote{\xdef\@thefnmark{}\@footnotetext}

\renewcommand{\mod}{\operatorname{mod}}

\input xy
\xyoption{all}
\usepackage{amssymb}

\def\-{^{-1}}

\begin{document}

\title[Spectral measures generated by arbitrary and random convolutions]{Spectral measures generated by arbitrary and random convolutions}
\author{Dorin Ervin Dutkay}

\address{[Dorin Ervin Dutkay] University of Central Florida\\
	Department of Mathematics\\
	4000 Central Florida Blvd.\\
	P.O. Box 161364\\
	Orlando, FL 32816-1364\\
U.S.A.\\} \email{Dorin.Dutkay@ucf.edu}

%

\author{Chun-Kit Lai}

\address{[Chun-Kit Lai]Department of Mathematics, San Francisco State University,
1600 Holloway Avenue, San Francisco, CA 94132.}

 \email{cklai@sfsu.edu}

\thanks{}
\subjclass[2010]{42B10,28A80,42C30}
\keywords{spectral measure, infinite convolution, tile, self-affine}

\begin{abstract}
We study spectral measures generated by infinite convolution products of discrete measures generated by Hadamard triples, and we present sufficient conditions for the measures to be spectral, generalizing a criterion by Strichartz. We then study the spectral measures generated by random convolutions of finite atomic measures and rescaling, where the digits are chosen from a finite collection of digit sets. We show that in dimension one, or in higher dimensions under certain conditions, ``almost all'' such measures generate spectral measures, or, in the case of complete digit sets, translational tiles. Our proofs are based on the study of self-affine spectral measures and tiles generated by Hadamard triples in quasi-product form.

\end{abstract}
\maketitle
\tableofcontents

\section{Introduction}

Let $\mu$ be a compactly supported Borel probability measure on
${\br}^d$ and let $\langle\cdot,\cdot\rangle$ and $\ip{\cdot}{\cdot}_{L^2(\mu)}$ denote respectively the standard inner product on $\br^d$ and $L^2(\mu)$. The measure $\mu$ is called a {\it spectral measure} if there
exists a countable set $\Lambda\subset {\mathbb R}^d$, called {\it spectrum} of the measure $\mu$, such  that the collection of exponential functions
$E(\Lambda): = \{e^{2\pi i \langle\lambda,x\rangle}:
\lambda\in\Lambda\}$ forms an orthonormal basis for $L^2(\mu)$. If we define the Fourier transform of $\mu$  to be
 $$
\widehat{\mu}(\xi)= \int e^{-2\pi i \langle\xi,x\rangle}d\mu(x),
 $$
then $E(\Lambda)$ is an orthonomal basis for $\mu$ if and only if
\begin{enumerate}
\item (Mutual orthogonality) $\widehat{\mu}(\lambda-\lambda')=0$ for all $\lambda\neq\lambda'\in\Lambda$.
\item (Completeness) If $\ip{f}{e^{2\pi i \langle\lambda,x\rangle}}_{L^2(\mu)}=0$ for all $\lambda\in\Lambda$, then $f=0$, $\mu$-a.e.
\end{enumerate}
If only condition (i) is satisfied, then we say that $\Lambda$ is a {\it mutually orthogonal set}. 
\medskip

Classical spectral measures were first introduced by Fuglede \cite{Fug74} when he studied his famous conjecture stating that $\chi_{\Omega}dx$ is a spectral measure if and only if $\Omega$ is a translational tile. Although the conjecture was proven to be false eventually in general \cite{Tao04,MR2237932}, this conjecture has generated a lot of interest (see \cite{LWtiling,LWspectral,IKT,IK,MR1700084,Ko1} and the reference therein) and it is related to the construction of Gabor and wavelet bases \cite{LW,Y}.   The studies entered into the realm of fractals when Jorgensen and Pedersen discovered that some singular fractal measures can also be spectral \cite{JP98}. Since then, singular spectral measures has been an active research topic which involves constructing new examples \cite{LaWa02,str00,DJ07d}, classifying classes of measures which are spectral \cite{MR2435649,Dai} and classifying their possible spectra \cite{DHS09,MR3055992}. It was surprising to find that the convergence of the associated Fourier series is uniform in the space of continuous functions \cite{MR2279556}. All the constructions of singular spectral measures, in the literature, to the best of our knowledge, are based on the Hadamard triple assumption.

\begin{definition}\label{hada}
Let $R\in M_d({\mathbb Z})$ be an $d\times d$ expansive matrix (expansive means that all eigenvalues have modulus strictly greater than 1) with integer entries. Let $B, L\subset{\mathbb Z}^d $ and $0\in B\cap L$ be  finite sets of integer vectors with $N:= \#B=\#L$ ($\#$ denotes the cardinality). We say that the system $(R,B,L)$ forms a {\it Hadamard triple}  if the matrix
\begin{equation}\label{Hadamard triples}
H=\frac{1}{\sqrt{N}}\left[e^{2\pi i \langle R^{-1}b,\ell\rangle}\right]_{\ell\in L, b\in B}
\end{equation}
is unitary, i.e., $H^*H = I$.
\end{definition}
Given a discrete set $A\subset {\mathbb R}^d$, we define the discrete measure on $A$ by
$$
\delta_{A} = \frac{1}{\#A}\sum_{a\in A}\delta_a
$$
where $\delta_a$ is the Dirac mass at $a$. From a direct observation, we can easily see that $(R,B,L)$ forms a Hadamard triple if and only if the discrete measure $\delta_{R^{-1}B}$ is a spectral measure with spectrum $L$. Singular spectral measures have been constructed by  infinite convolutions of these discrete measures. To put it in the most general sense, suppose that we are given a sequence of Hadamard triples $(R_i,B_i,L_i)$, $i=1,2,...$. Then we define
$$
{\bf R}_n = R_n...R_1
$$
 and the probability measure induced by these triples as
\begin{equation}\label{mu}
\mu = \mu(R_i,B_i) = \delta_{{\bf R}_1^{-1}B_1}\ast\delta_{{\bf R}_2^{-1}B_2}\ast...\ast\delta_{{\bf R}_n^{-1}B_n}\ast....,
\end{equation}
assuming the infinite convolution product is weakly convergent to a Borel probability measure.  

It is easy to  show that the measure has an infinite mutually orthogonal set
\begin{equation}\label{Lambda}
\Lambda = L_1+R_1^TL_2+...+(R_1^TR_2^T...R_{n-1}^T)L_n+...
\end{equation}
The spectral property of these measures was first studied by Strichartz \cite{str00}, in which the sequence $\{(R_i,B_i)\}$ was called a {\it compatible tower}, and it has received a lot of attention recently since all measures arising from factorization of Lebesgue measure on $[0,1]^d$ are of this type \cite{GL} and it gives rise to spectral measures with support of arbitrary dimensions \cite{MR3318656}. We also note that if all $R_i  =  R$ and $B_i = B$ for some  expanding matrix $R\in M_d({\mathbb Z})$   and $B\subset{\mathbb R}^d$, then the measure
\begin{equation}\label{eqself-affine}
\mu =\mu_{R,B} = \delta_{{R}^{-1}B}\ast\delta_{{R}^{-2}B}\ast...\ast\delta_{{R}^{-n}B}\ast....
\end{equation}
  is reduced to the {\it self-affine measure} generated by the maps $\tau_b(x) = R^{-1}(x+b)$, see \cite{Hut81}. It was recently proved by the authors,  by suitably modifying $L$, that all self-affine  measures generated by Hadamard triples are spectral measures \cite{DL15,DHL15}. In this paper, we study the spectral property of these arbitrary convolution and then a special case of random convolution with finitely many choices of Hadamard triples, chosen in a random order.

  \medskip

  \noindent{\bf Arbitrary convolutions.}  We first generalize the Strichartz criterion for $\Lambda$ in (\ref{Lambda}) to be a spectrum for $\mu$. For $\mu$ and  $\Lambda$ in (\ref{mu}) and (\ref{Lambda}), we define
  $$
  \mu_n = \delta_{{\bf R}_1^{-1}B_1}\ast\delta_{{\bf R}_2^{-1}B_2}\ast...\ast\delta_{{\bf R}_n^{-1}B_n}, \ \mu_{>n} = \delta_{{\bf R}_{n+1}^{-1}B_{n+1}}\ast\delta_{{\bf R}_{n+2}^{-1}B_{n+2}}\ast...
  $$
  and
  $$
 K_{n} = \left\{\sum_{k=n+1}^{\infty}{\bf R}_k^{-1}b_k: b_k\in B_k \right\},  {\bf B}_n = \left\{\sum_{k=1}^{n}{\bf R}_k^{-1}b_k: b_k\in B_k \right\}.
  $$
 Hence,  $K_0 = \bigcup_{{\bf b}\in {\bf B}_n}({\bf b}_n+K_n)$ and $K_0, B_n, K_n$ are respectively the support of $\mu,\mu_n$ and $\mu_{>n}.$ We will also use the notation $T(\{R_i,B_i\})$ for the support $K_0$ of the measure $\mu$.

We say that $\mu$ satisfies the {\it no overlap condition} if
 $$
 \mu (({\bf b}_n+ K_n)\cap({\bf b}_n'+K_n))=0, \ \mbox{for all} \ {\bf b}_n\neq{\bf b}_n'\in{\bf B}_n, \ \mbox{for all}  \ n\in{\mathbb N}.
 $$
For $\Lambda$ in (\ref{Lambda}),  we also define its $n^{\rm th}$-level approximation.
$$
\Lambda_n = L_1+R_1^TL_2+...+R_1^TR_2^T...R_{n-1}^TL_n.
$$
It is easy to see that $\#{\bf B}_n = \#\Lambda_n = \prod_{i=1}^{n}N_i := {\bf M}_n$.
From this, we consider the following matrices
$$
{\mathcal F}_n = \frac{1}{\sqrt{{\bf M}_n}}\left[|\widehat{\mu_{>n}}(\lambda)|e^{-2\pi i \langle{\bf b}, \lambda\rangle}\right]_{\lambda\in \Lambda_n,{\bf b}\in {\bf B}_n}.
$$
Recall that the {\it singular values} of ${\mathcal F}_n$ are the eigenvalues of ${\mathcal F}_n^{\ast}{\mathcal F}_n$ and we denote by $\sigma({\mathcal F}_n)$ the set of all singular values of ${\mathcal F}_n.$
\medskip

Our first main result is as follows:
\begin{theorem}\label{th1.1}
Suppose that the measure $\mu$ in (\ref{mu}) satisfies the no-overlap condition and that $\mu$ is compactly supported.  If $\inf_n\min\sigma({\mathcal F}_n)>0$, then $\Lambda$ is a spectrum for $\mu$.
\medskip

In particular, if $\inf_n\inf_{\lambda\in\Lambda_n}|\widehat{\mu_{>n}}(\lambda)|>0$, then $\Lambda$ is a spectrum for $\mu$.
\end{theorem}

We remark that the assumption that  the measure $\mu$ is compactly supported ensures that the family of step functions is dense in $L^2(\mu)$ and the no-overlap condition is also necessary to ensure that  $\mu(K_{n}) = 1/{\bf M}_n$. In fact, if the no-overlap condition is not satisfied, $\mu$ can be non-spectral (see Example \ref{ex1.5}).

\medskip

\noindent{\bf Random Convolutions.} In the second part of the paper, we consider $R_i =R$ for all $i$ with $R$ is a fixed integral expanding matrix. Let also $B(1),....,B(N)$ be a finite collection of sets in ${\mathbb R}^d$, with $0\in B(i)$ and $\#B(i) = M (\le |\det R|)$, for all $i$, so that $(R,B(i),L)$ form Hadamard triples for all $i$.  Note that the set $L$ is the same for all $i$.

Let $\omega=\omega_1\omega_2\dots$ be an infinite word in $\{1,\dots, N\}^{\bn}$. The measure $\mu$ in (\ref{mu}) is now read as a random convolution of discrete measures scaled by $R$.
\begin{equation}\label{mu_omega}
\mu_\omega=\mu(\omega,R):=\delta_{R^{-1}B(\omega_1)}*\delta_{R^{-2}B(\omega_2)}*\dots.
\end{equation}
Some special cases of these measures were studied by He et al. \cite{AHLau,AHL}. We will see that spectral measures exist in abundance in the setting of random convolutions. To be precise, we treat $\omega_n$ as independent random variables with values $1,...,N$, with equal probability $1/N$, and ${\mathbb P}$ is the product probability on $\{1,...,N\}^{\mathbb N}$. i.e.
\begin{equation}\label{P}
 {\mathbb P} (\omega_{1}= i_1,...,\omega_{k}= i_k) = \frac{1}{N^k}, \ \forall k\in{\mathbb N}, \ i_1,..,i_k\in\{1,...,N\}.
 \end{equation}
 The main important observation is that measures in (\ref{mu_omega}) can be put together in the fibres of the self-affine measures generated by a Hadamard triple in quasi-product form.

\begin{definition}\label{def1.5} Given the Hadamard triples $(R,B(i),L)$, $i=1,...,N$ and $\#B(i) = M$. We associate the matrix ${\bf R}$ and the sets ${\bf B}$ and ${\bf L}$ with the following form:
\beq
 {\bf R}=\begin{bmatrix} R_1&0\\ C &R\end{bmatrix} ,
\eeq
where $R_1 \in {M}_r(\bz)$, $R \in {M}_{d} (\bz)$ and $C\in M_{d, r}(\bz)$. Let
\beq
{\bf B} = \left\{\ \begin{bmatrix}a_i \\ d_{i,j} \end{bmatrix} : i \in \{1,...,N \}, d_{i,j} \in B(i) \right\}\,
\eeq
where $a_i \in \bz^r$, $d_{i,j} \in \bz^{d}$ and $a_1=d_{i,1}=0$, for all $i,j$.

Suppose ${\bf L}= L_1 \times L$ with $L_1 \subset \bz^r$, $L \subset \bz^{d}$ and $(R_1, B_1 := \{ a_i : 1\leq i\leq N \} , L_1 )$ is a Hadamard triple. Then we say that $({\bf R},{\bf B},{\bf L})$ is in  {\it quasi-product form} on ${\mathbb R}^{d+r}$ associated with  $(R,B(i),L)$, $i=1,...,N$. The self-affine measure associated with $({\bf R},{\bf B},{\bf L})$ is the measure defined by
$$
\mu_{{\bf R},{\bf B}} = \delta_{{\bf R}^{-1}{\bf B}}\ast\delta_{{\bf R}^{-2}{\bf B}}\ast...\ast\delta_{{\bf R}^{-n}{\bf B}}\ast....
$$
We denote also by  $\mu_1$  the self-affine measure associated with  $(R_1, B_1)$ defined in (\ref{eqself-affine}).

\end{definition}

\medskip

\begin{theorem}\label{pr1.3}
 Let $(R,B(i),L)$, $i=1,...,N$, be the Hadamard triples and $({\bf R},{\bf B},{\bf L})$ be the triple in quasi-product form associated with $(R,B(i),L)$ in Definition \ref{def1.5}. Assume $\Lambda_1$ is a spectrum for $\mu_1$ and let $\Lambda_2$ be a subset of $\br^{d}$. Then $\Lambda_1\times \Lambda_2$ is a spectrum for $\mu_{{\bf R},\bf B}$ if and only if $\Lambda_2$ is a spectrum for $\mu_{\omega}$, ${\mathbb P}$-almost surely.
\end{theorem}

\medskip

With the theorem above, we will construct a spectrum of the form $\Lambda_1\times \Lambda_2$ for $\mu_{{\bf R},\bf B}$ in some associated quasi-product form. Under two different assumptions, we have the following conclusion:

\begin{theorem}\label{th1.4}
 Let $(R,B(i),L)$, $i=1,...,N$, be the Hadamard triples. Assume that one of the following condition holds:
 \begin{enumerate}
 \item the Hadamard triples $(R,B(i),L)$ are on ${\mathbb R}^1$, i.e. $R$ is an integer.
 \item Each $B(i)$ is a complete set representative of $R$
  \end{enumerate}
  Then there exists a set $\Lambda$ such that $\Lambda$ is a spectrum for $\mu_{\omega}$, for ${\mathbb P}$-almost every $\omega$.
	
\medskip

	Moreover, in the case (ii), there exists a lattice $\tilde \Gamma$ such that
	the support $T(\{R,B(i_k)\}_k)$ of the measure $\mu_\omega$ tiles $\br^d$ by $\tilde\Gamma$ and $\mu_\omega$ is the normalized Lebesgue measure on $T(\{R,B(i_k)\})$, for $\mathbb P$-almost every $\omega=(i_1i_2\dots)$.
\end{theorem}

\begin{definition}
We say that a Lebesgue measurable set $T$ tiles $\br^d$ by a set $\mathcal T$, if $(T+t)_{t\in\mathcal T}$ is a partition of $\br^d$, up to Lebesgue measure zero.
\end{definition}
\medskip
 The proof for the first case involves one of the canonical spectra in spectral measures theory. These are studied in \cite{DJ06,DJ07d,DJ09}. We call it here the {\it dynamically simple spectra} (Definition \ref{definv2}). We will summarize this in a separate study in the appendix of this paper. Theorem \ref{th1.4} perhaps hints towards a conjecture about random convolutions.

\begin{conjecture}
Let $(R,B(i),L)$, $i=1,...,N$, be the Hadamard triples on ${\mathbb R}^d$.   Then some associated quasi-product form admits a spectrum of the form $\Lambda_1\times\Lambda_2$ and hence $\Lambda_2$ is a spectrum for $\mu_{\omega}$, ${\mathbb P}$-almost surely.
\end{conjecture}
Theorem \ref{th1.4} showed that the conjecture is true on ${\mathbb R}^1$ and  in the case when we can construct a quasi-product form self-affine tile. In the end of the introduction, we illustrate Theorem \ref{th1.4} by an example. It is very interesting to notice, that some simple infinite convolution products, are not spectral. This sheds some light on our results that show that ``almost every'' infinite convolution is a spectral measure. However, not all of them as we see in the next example. 

\begin{example}\label{ex1.5}
Let $R=2$ and $B(0) = \{0,1\}$ and $B(1) = \{0,3\}$. As each $B(i)$ is a complete residue modulo $2$. Theorem \ref{th1.4} shows that, almost surely,
$$
\mu_{\omega} = \delta_{B(\omega_1)/2}\ast\delta_{B(\omega_2)/2^2}\ast...
$$
is a spectral measure with a common spectrum ${\mathbb Z}$.  However, if we consider a special case with $\omega = 01111...$, we see that the measure
$$
\mu_{\omega} = \delta_{\{0,1\}/2}\ast {\mathcal L}_{[0,3/2]},
$$
where ${\mathcal L}_{[0,3/2]}$ is the normalized Lebesgue measure supported on the interval $[0,3/2]$. Thus, in the first level, the no-overlap condition is not satisfied. Moreover, the measure $\mu_{\omega}$ is absolutely continuous with respect to the Lebesgue measure, but it is not spectral as the density is not uniformly distributed \cite{DL14}. Despite this specific example, the measures $\mu_{\omega}$ are spectral, for almost all $\omega\in\{0,1\}^{\bn}$, by Theorem \ref{th1.4}.
\end{example}

One may also refer to \cite{AHLau, AHL} for some deterministic examples in which the random convolution is spectral everywhere. However, strong assumption on $L$ is required and it does not cover Example \ref{ex1.5}.

\medskip

We organize our paper as follows: we study arbitrary convolutions in Section 2 and random convolutions in Section 3. In the appendix, we study the dynamically simple spectrum used in Section 3.

\section{Arbitrary convolutions}

Given a sequence of Hadamard triples $\{(R_i,B_i,L_i)\}$ with measures $\mu$ defined in (\ref{mu}), its Fourier transform is easily computed as
$$
\widehat{\mu}(\xi)  = \prod_{n=1}^{\infty} \widehat{\delta_{B_i}} (({\bf R}_n^{T})^{-1}\xi).
$$
We first note that
\begin{lemma}
The set $\Lambda$ in (\ref{Lambda}) is a mutually orthogonal set for $\mu$.
\end{lemma}

\begin{proof}
This was proved in Strichartz \cite[Theorem 2.7]{str00}. In short, it follows from the fact that the Hadamard matrices $H_n = \frac{1}{\sqrt{N_i}}\left[e^{2\pi i \langle R^{-1}b,\ell\rangle}\right]_{\ell\in L_i, b\in B_i}$ have mutually orthogonal rows, and so does the matrix 
$$\frac{1}{\sqrt{{\bf M}_n}}\left(e^{-2\pi i \ip{{\bf R}^{-n}\bf b}{\lambda}}\right)_{\lambda\in\Lambda_n,\bf b\in {\bf B}_n}.$$
\end{proof}

\medskip

Recall that we can write the support of $\mu$, $K_{0}$, as
\begin{equation}\label{K_0}
K_{0} = \bigcup_{{\bf b}\in {\bf B}_n} ({\bf b}+ K_{n}).
\end{equation}
Denote by $K_{\bf b} = {\bf b}+ K_{n}$ and by ${\bf 1}_{K_{\bf b}}$ the characteristic function of $K_{\bf b}$. Let
$$
{\mathcal S}_n = \left\{\sum_{{\bf b}\in {\bf B}_n}w_{\bf b}{\bf 1}_{K_{\bf b}}: w_{\bf b}\in{\mathbb C}\right\}.
$$
${\mathcal S}_n$ denotes the collection of all $n^{th}$ level step functions on $K_{0}$. As
$$
K_{n} = \bigcup_{b\in B_{n+1}}\left({\bf R}_{n+1}^{-1}{b}+K_{n+1}\right)
$$
and $0\in B_{n}$ for all $n$, we have ${\mathcal S}_1\subset {\mathcal S}_2\subset....$. Let also
$$
{\mathcal S} = \bigcup_{n=1}^{\infty}{\mathcal S}_n.
$$

\begin{lemma}\label{lem3.0}
If $\mu$ is compactly supported, then  ${\mathcal S}$ forms a dense set of functions in $L^2(\mu)$.
\end{lemma}

\begin{proof}
Take first a continuous function $f$ on $K_0$
and $\epsilon > 0$. Since $K_0$ is compact, the function $f$ is uniformly continuous. We can find $m$ large enough such that the diameter of all sets $K_{\bf b}$, ${\bf b}\in {\bf B}_m$, is small enough so that $ |f(x) - f(y)| <\epsilon$ for all
$x,y\in K_{\bf b}$. Consider $g = \sum_{{\bf b}\in {\bf B}_m}f({\bf b}){\bf 1}_{K_{\bf b}}$. It is easy to see that $\sup_{x\in K_0}|f(x)-g(x)|<\epsilon$. Hence, ${\mathcal S}$ is uniformly dense in $C(K_0)$. As $\mu$ is a regular Borel measure, ${\mathcal S}$ is dense in $L^2(\mu).$
\end{proof}

\medskip

\begin{lemma}\label{lem3.1}
Let $f = \sum_{{\bf b}\in {\bf B}_n}w_{\bf b}{\bf 1}_{K_{\bf b}}\in {\mathcal S}_n$ and let ${\bf w} = (w_{\bf b})_{{\bf b}\in{\bf B}_n}$. Denote by $\|\cdot\|$ the Euclidean norm on ${\mathbb C}^{{\bf M}_n}$. Then
\begin{equation}\label{eq3.1}
\int|f|^2d\mu = \frac{1}{{\bf M}_n}\sum_{{\bf b}\in {\bf B}_n}|w_{\bf b}|^2 =  \frac{1}{{\bf M}_n}\|w_{\bf b}\|^2 .
\end{equation}
\begin{equation}\label{eq3.2}
\int f(x)e^{-2\pi i \ip{\lambda}{x}}d\mu(x) = \frac{1}{{\bf M}_n}\widehat{\mu_{>n}}(\lambda)\sum_{{\bf b}\in {\bf B}_n}w_{\bf b} e^{-2\pi i  \ip{{\bf b}}{ \lambda}}.
\end{equation}
(Recall that ${\bf M}_n = N_1...N_n$). Moreover,
\begin{equation}\label{eq3.3}
\sum_{\lambda\in\Lambda_n}\left|\int f(x)e^{-2\pi i \ip{\lambda}{x}}d\mu(x)\right|^2=\frac{1}{{\bf M}_n}\|{\mathcal F}_n{\bf w}\|^2.
\end{equation}
\end{lemma}

\begin{proof}
Note that
  \begin{equation}\label{eq3.0}
   \mu(K_{\bf b}) = \int {\bf 1}_{K_{\bf b}}(x)d(\mu_n\ast\mu_{>n}(x)) = \frac{1}{{\bf M}_n}+\frac{1}{{\bf M}_n}\sum_{{\bf b}'\in{\bf B}_n, {\bf b}'\neq{\bf b}}\int  {\bf 1}_{{\bf b}+K_n}({\bf b}'+y)d\mu_{>n}(y).
  \end{equation}
  This implies that $\mu(K_{\bf b})\ge 1/{\bf M}_n$ for all ${\bf b}\in {\bf B}_n$. On the other hand, because of the no-overlap condition and (\ref{K_0}),
$$
1 = \mu(K_0) = \mu\left(\bigcup_{{\bf b}\in {\bf B}_n} K_{\bf b}\right) = \sum_{{\bf b}\in {\bf B}_n} \mu(K_{\bf b}).
$$
If $\mu(K_{\bf b})> 1/{\bf M}_n$ for some ${\bf b}\in {\bf B}_n$, then $\sum_{{\bf b}\in {\bf B}_n} \mu(K_{\bf b})>1$, which is a contradiction. Hence,  all $K_{\bf b}$, ${\bf b}\in {\bf B}_n$ have the same $\mu$-measure $1/{\bf M}_n$ and (\ref{eq3.1}) follows from a direct computation. For (\ref{eq3.2}), we  note that (\ref{eq3.0}) now becomes
$$
\frac{1}{{\bf M}_n} = \mu(K_{\bf b}) = \int {\bf 1}_{K_{\bf b}}(x)d(\mu_n\ast\mu_{>n}(x)) = \frac{1}{{\bf M}_n}+\frac{1}{{\bf M}_n}\sum_{{\bf b}'\in{\bf B}_n, {\bf b}'\neq{\bf b}}\int  {\bf 1}_{{\bf b}+K_n}({\bf b}'+y)d\mu_{>n}(y)
$$
since supp $\mu_{>n} = K_{n}$. Thus, $\int  {\bf 1}_{{\bf b}+K_n}({\bf b}'+y)d\mu_{>n}(y)=0$ and ${\bf 1}_{{\bf b}+K_n}({\bf b}'+y)=0$ $\mu_{>n}$-a.e. Hence,
$$
\begin{aligned}
\int f(x)e^{-2\pi i \ip{\lambda}{x}}d\mu(x) = &\sum_{{\bf b}\in{\bf B}_n}w_{\bf b} \int {\bf 1}_{K_{\bf b}}(x)e^{-2\pi i \ip{\lambda}{x}}d(\mu_n\ast\mu_{>n}(x)) \\
=& \sum_{{\bf b}\in{\bf B}_n}w_{\bf b} \int\int {\bf 1}_{{\bf b}+K_{n}}(x+y)e^{-2\pi i \ip{\lambda}{x+y}}d\mu_n(x)d\mu_{>n}(y).\\
=&\sum_{{\bf b}\in{\bf B}_n}w_{\bf b} \frac{1}{{\bf M}_n}\int \sum_{{\bf b}'\in{\bf B}_n}{\bf 1}_{{\bf b}+K_{n}}({\bf b}'+y)e^{-2\pi i \ip{\lambda}{{\bf b}'+y}}d\mu_{>n}(y)\\
=&\sum_{{\bf b}\in{\bf B}_n}w_{\bf b} \frac{1}{{\bf M}_n}e^{-2\pi i \ip{\lambda}{{\bf b}}}\int e^{-2\pi i \ip{\lambda}{ y}}d\mu_{>n}(y)\\
=& \frac{1}{{\bf M}_n}\widehat{\mu_{>n}}(\lambda)\sum_{{\bf b}\in {\bf B}_n}w_{\bf b} e^{-2\pi i  \ip{{\bf b}}{\lambda}}.\\
\end{aligned}
$$
Thus (\ref{eq3.2}) follows. Finally, we have
\begin{equation}\label{eq3.4}
\begin{aligned}
\sum_{\lambda\in\Lambda_n}\left|\int f(x)e^{-2\pi i \ip{\lambda}{x}}d\mu(x)\right|^2 = &\frac{1}{{\bf M}_n} \sum_{\lambda\in\Lambda_n}|\widehat{\mu_{>n}}(\lambda)|^2\left|\sum_{{\bf b}\in {\bf B}_n}w_{\bf b}e^{-2\pi i  \ip{{\bf b}}{\lambda}}\right|^2\\
=&\frac{1}{{\bf M}_n} \sum_{\lambda\in\Lambda_n}\left|\sum_{{\bf b}\in {\bf B}_n}w_{\bf b}|\widehat{\mu_{>n}}(\lambda)|e^{-2\pi i  \ip{{\bf b}}{\lambda}}\right|^2 = \frac{1}{{\bf M}_n} \|{\mathcal F}_n{\bf w}\|^2,
\end{aligned}
\end{equation}
and (\ref{eq3.3}) follows.
\end{proof}

\medskip

We are now ready to prove our first theorem. We recall a standard fact of matrix analysis: If $A$ is a self-adjoint matrix and $\lambda_{\min}$ is its minimum eigenvalue, then
$$
\lambda_{\min} = \min_{\|{\bf w}\|=1}\langle A{\bf w},{\bf w}\rangle.
$$

\begin{proof}[Proof of Theorem \ref{th1.1}] Suppose first that  $\sigma: = \inf_n \min\sigma({\mathcal F}_n)>0$. For all $f\in {\mathcal S}_n$, by equations (\ref{eq3.1}) and (\ref{eq3.3}) in Lemma \ref{lem3.1},
$$
\sum_{\lambda\in\Lambda_n}\left|\int f(x)e^{-2\pi i \ip{\lambda}{x}}d\mu(x)\right|^2= \frac{1}{{\bf M}_n} \|{\mathcal F}_n{\bf w}\|^2 = \frac{1}{{\bf M}_n}\langle{\mathcal F}_n^{\ast}{\mathcal F}_n{\bf w},{\bf w}\rangle\geq \frac{1}{{\bf M}_n}\sigma\|{\bf w}\|^2 = \sigma\int|f|^2d\mu.
$$
As $f\in {\mathcal S}_n\subset{\mathcal S}_m$ for all $m>n$, we apply  $f$ to the inequality for $m$ and obtain
$$
\sum_{\lambda\in\Lambda_m}\left|\int f(x)e^{-2\pi i \ip{\lambda}{x}}d\mu(x)\right|^2\ge\sigma\int|f|^2d\mu.
$$
Taking $m$ to infinity and using the fact that $\Lambda = \bigcup_{m=1}^{\infty}\Lambda_m$, we have
$$
\sum_{\lambda\in\Lambda}\left|\int f(x)e^{-2\pi i \ip{\lambda}{x}}d\mu(x)\right|^2\ge\sigma\int|f|^2d\mu.
$$
As ${\mathcal S}$ forms a dense set, the above inequality is actually true for any $f\in L^2(\mu)$. This establishes the completeness of $\Lambda$ in $L^2(\mu)$.

\medskip

Let $\delta = \inf_n\inf_{\lambda\in \Lambda_n}|\widehat{\mu_{>n}}(\lambda)|$. We now prove the special case.  This follows from a direct observation that
$$
\|{\mathcal F}_n{\bf w}\|^2 \geq \delta^2 \|{\bf H}_n{\bf w}\|^2 = \delta^2\|{\bf w}\|^2
$$
where ${\bf H}_n  =\frac{1}{\sqrt{{\bf M}_n}}\left[e^{-2\pi i \langle{\bf b}, \lambda\rangle}\right]_{\lambda\in \Lambda_n,{\bf b}\in {\bf B}_n} $ is a Hadamard matrix. Hence, $\sigma\geq \delta^2>0$.
\end{proof}

\medskip

\begin{remark}We note that the theorem generalizes the result of Strichartz \cite[Theorem 2.8]{str00}, which asserted that if the Hadamard triples $(R_i,B_i, L_i)$ are chosen only from {\it finitely many choices}, and the zero sets $Z_i$ of the functions
$$m_{B_i}(x)=\frac{1}{N_i}\sum_{b\in B_i}e^{2\pi i \ip{b}{x}}$$  are separated from the set
$$
\Gamma_n = ({\bf R}_n^T)^{-1} \left(L_1+R_1^TL_2+...+(R_1^TR_2^T...R_{n-1}^T)L_n\right)
$$
by a distance $\delta>0$, uniformly in $n$. Then the measure $\mu(R_i,B_i)$ is a spectral measure. Indeed, this assumption implies $\inf_n\inf_{\lambda\in\Lambda_n}|\mu_{>n}(\lambda)|^2>0$. To see this, we note that
$$
|\widehat{\mu_{>n}}(\lambda)|^2 = \prod_{k=n+1}^{\infty}|m_{B_k} (({\bf R}_k^T)^{-1}\lambda)|^2.
$$
As there are only finitely many $B_i$, $\Gamma_n$ lies inside a compact set independent of $n$. From the fact that $m_{B_i}(0)=1$ and that $({\bf R}_k^T)^{-1}\lambda$ decays to zero exponentially,  we can find a $k_1$, independent of $\lambda\in\Lambda_n$ such that $\prod_{k=n+k_1+1}^{\infty}|m_{B_k} (({\bf R}_k^T)^{-1}\lambda)|^2$ is uniformly bounded below by some constant $c>0$. For the first $k_1$ terms, the assumption on $Z_i$ guarantees they are bounded away from $\delta'^{k_1}$, for some $\delta'>0$. Thus,
$$
\inf_n\inf_{\lambda\in\Lambda_n}|\widehat{\mu_{>n}}(\lambda)|^2 \geq \delta'^{k_1}c>0.
$$
\end{remark}
\medskip
\section{Random convolution}

In this section, we study random convolutions of discrete measures generated by Hadamard triples. We first show that quasi-product forms generate Hadamard triples.

\begin{proposition}\label{pr1.1}
If $({\bf R},{\bf B},{\bf L})$ is in quasi-product form as in Definition \ref{def1.5}, then  $({\bf R},{\bf B},{\bf L})$ is a Hadamard triple on ${\mathbb R}^{d+r}$.
\end{proposition}

\begin{proof}
We have that ${\bf R}^{-1}$ is of the form
\beq
 {\bf R}^{-1}=\begin{bmatrix} R_1^{-1}&0\\ D &R^{-1}\end{bmatrix} ,
\eeq
for some matrix $D$. Consider ${\bf b} = \begin{bmatrix} a_i\\ d_{i,j}\end{bmatrix}\neq{\bf b}' = \begin{bmatrix} a_{i'}\\ d_{i',j'}\end{bmatrix}$ and
$$
\begin{aligned}
A((i,j),(i',j')):=&\sum_{{\bf \ell}\in {\bf L}}e^{-2\pi i \ip{{\bf R}^{-1}({\bf b}-{\bf b}')}{{\bf \ell}}}\\
=&\sum_{\ell_1\in L_1}\sum_{\ell_2\in L}e^{-2\pi i \left(\ip{R_1^{-1}(a_i-a_{i'})}{ \ell_1}+\ip{D(a_i-a_{i'})}{\ell_2}+\ip{R^{-1}(d_{i,j}-d_{i',j'})}{\ell_2}\right)
}\\
=&\left(\sum_{\ell_1\in L_1}e^{-2\pi i \ip {R_1^{-1}(a_i-a_i')}{ \ell_1}}\right)\cdot\left(\sum_{\ell_2\in L_2}e^{-2\pi i \left( \ip{D(a_i-a_{i'})}{ \ell_2}+\ip{R_2^{-1}(d_{i,j}-d_{i',j'})}{ l_2}\right)}\right).\end{aligned}
$$
If $i\neq i'$ then $A((i,j),(i',j'))=0$ because $(R_1,B_1,L_1)$ is a Hadamard triple. If $i=i'$, then
$$A((i,j),(i',j'))=N_1\sum_{l_2\in L_2}e^{2\pi i R_2^{-1}(d_{i,j}-d_{i,j'})\cdot l_2}=0,$$
because $(R_2,B_2(i),L_2)$ are  Hadamard triples for all $i$. This shows that the matrix $\left[e^{-2\pi i \ip{{\bf R}^{-1}{\bf b}}{{\bf\ell}}}\right]_{{\bf \ell}\in{\bf L},{\bf b}\in{\bf B}}$ has mutually orthogonal rows and hence  $({\bf R},{\bf B},{\bf L})$ is a Hadamard triple on ${\mathbb R}^{d+r}$.
\end{proof}

\medskip

We now derive and collect the necessary information about the self-affine measure generated by the quasi-product form  in Definition \ref{def1.5}. These properties were all considered in \cite{DJ07d}. First, we note that
$${\bf R}^{-1}=\begin{bmatrix}
R_1^{-1}&0\\
-R^{-1}CR_1^{-1}&R^{-1}
\end{bmatrix}$$
and, by induction,
$${\bf R}^{-k}= \begin{bmatrix}
R_1^{-k}&0\\
D_k&R^{-k}
\end{bmatrix},\mbox{ where }D_k:=-\sum_{j=0}^{k-1}R^{-(j+1)}CR_1^{-(k-j)}.$$
The support of the self-affine measure $\mu$ defined by ${\bf R}$ and $ {\bf B}$ is given by
 \begin{equation}\label{T(R,B)}
 T({\bf R},{\bf B})=\left\{\sum_{k=1}^\infty {\bf R}^{-k}{\bf b}_k : {\bf b}_k\in {\bf B}\right\}.
 \end{equation}
Therefore any element $(x,y)^T\in T({\bf R},{\bf B})$ can be written in the following form
$$x=\sum_{k=1}^\infty R_1^{-k}a_{i_k}, \quad y=\sum_{k=1}^\infty D_ka_{i_k}+\sum_{k=1}^\infty R^{-k}d_{i_k,j_k}.$$
Let $X_1$ be the attractor (in $\br^r)$ associated to the IFS defined by the pair $(R_1,B_1)$, i.e.,
 $$X_1=T(R_1,B_1) = \left\{\sum_{k=1}^\infty { R}^{-k}{b}_k : { b}_k\in { B}_1\right\}.
 $$ Let $\mu_1$ be the (equal-weighted) invariant measure associated to this pair.

\medskip

For each sequence $\omega=(i_1i_2\dots)\in\{1,\dots,N\}^{\bn} = \{1,\dots, N\}\times\{1,\dots, N\}\times...$, define the map $\pi: \Omega_1\rightarrow X_1$ by
\begin{equation}\label{eqxomega}
\pi(\omega)=\sum_{k=1}^\infty R_1^{-k}a_{i_k}.
\end{equation}
 As $(R_1,B_1)$ forms a Hadamard triple with  $L_1$, the measure $\mu_1$ has the no-overlap property \cite[Theorem 1.7]{DL15}. It implies that for $\mu_1$-a.e. $x\in X_1$, there is a unique $\omega$ such that $\pi(\omega)=x$. We define this as $\pi^{-1}(x)$. This establishes a bijective correspondence, up to measure zero, between the set $\Omega_1:=\{1,\dots,N\}^{\bn}$ and $X_1$. For details about the correspondence, one can refer to \cite[Section 1.4]{Ki}. The measure $\mu_1$ from $X_1$ is pulled back to the product measure ${\mathbb P}$ defined in (\ref{P}), i.e.
 $$
\mu_1= {\mathbb P}\circ \pi^{-1}.
 $$
\medskip

For $\omega=(i_1i_2\dots)$ in $\Omega_1$, define
$$\Omega_2(\omega):=\{(d_{i_1,j_1}d_{i_2,j_2}\dots d_{i_n,j_n}\dots) : j_k\in \{1,\dots,M\}\}.$$
For $\omega\in\Omega_1$, define $g(\omega):=\sum_{k=1}^\infty D_ka_{i_k}$.
   Also  define
$$X_2(\omega):=\left\{\sum_{k=1}^\infty R_2^{-k}d_{i_k,j_k}: d_{i_k,j_k}\in B(i_k)\right\}.$$
Note that $T({\bf R},{\bf B})$ takes the following form

$$T({\bf R},{\bf B})=\{(\pi(\omega),g(\omega)+y)^T: \omega\in \Omega_1,y\in X_2(\omega)\}.$$

\medskip

For $x\in X_1$, up to a $\mu_1$-measure zero set, $F$, we can write $\omega = \pi^{-1}(x) = (i_1i_2...)$ and we can define $\mu_x^2$ to be the infinite convolution product defined by $\mu_{\omega}$ in (\ref{mu_omega}). i.e.
$$
\mu_{x}^2 = \mu_{\omega} =\delta_{R_2^{-1}B_2(i_1)}\ast\delta_{R_2^{-2}B_2(i_2)}\ast\dots.
$$
with the support of $\mu_x^2$ equal to $X_2(x) := X_2(\pi^{-1}(x))$.

\medskip

 The following lemmas, established in \cite{DJ07d}, are the key identities for our analysis.

\begin{lemma}\label{lem1.23}\cite[Lemma 4.4]{DJ07d}
For any bounded Borel functions on $\br^d$, the self-affine measure $\mu_{\bf R,\bf B}$ satisfies
$$\int_{T(R,B)}f\,d\mu_{{\bf R},{\bf B}}=\int_{X_1}\int_{X_2(x)}f(x,y+g(x))\,d\mu_x^2(y)\,d\mu_1(x).$$
\end{lemma}

\begin{lemma}\label{lem1.24}\cite[Lemma 4.5]{DJ07d}
If $\Lambda_1$ is a spectrum for the measure $\mu_1$, then
$$F(y):=\sum_{\lambda_1\in\Lambda_1}|\widehat{\mu_{{\bf R},{\bf B}}}(x+\lambda_1,y)|^2=\int_{X_1}|\widehat\mu_s^2(y)|^2\,d\mu_1(s),\quad(x\in\br^r,y\in\br^{d-r}).$$
\end{lemma}

We recall also the Jorgensen-Pedersen Lemma about checking when $\Lambda$ is a spectrum for $\mu$.

\begin{lemma}\label{JP}\cite{JP98}
$\Lambda$ is a spectrum for a probability measure $\mu$  on ${\mathbb R}^d$ if and only if
$$
Q(\xi): = \sum_{\lambda\in\Lambda}|\widehat{\mu}(\xi+\lambda)|^2\equiv1.
$$
Moreover, if $\Lambda$ is an orthogonal set, then $Q$ is an entire function on ${\mathbb C}^d$ with $0\leq Q(x)\leq 1$ for $x\in\br^d$.
\end{lemma}


\medskip

\begin{proof}[Proof of Theorem \ref{pr1.3}] Assume that $\Lambda_2$ is a spectrum for $\mu_\omega$ for ${\mathbb P}$-a.e. $\omega$. Let $$
E = \{\omega: \mbox{$\Lambda_2$ is a spectrum for $\mu_\omega$}\}.
$$
 Then ${\mathbb P}(E\cap(\Omega_1\setminus F))=1$. The set of points $x\in\Omega_1$ such that   $\Lambda_2$ is a spectrum for $\mu_x^2$ is exactly equal to $\pi^{-1}(x)\in E\cap (\Omega_1\setminus F)$. This shows $\Lambda_2$ is a spectrum for $\mu_x^2$ for $\mu_1$-a.e. $x$.
 We now check that $\Lambda_1\times\Lambda_2$ is a spectrum for $\mu_{{\bf R},{\bf B}}$. For $(x,y)\in\br^{d+r}$ we have , with Lemma \ref{lem1.24} and Fubini's theorem,
$$\sum_{\lambda_1\in\Lambda_1}\sum_{\lambda_2\in\Lambda_2}|\widehat\mu_{{\bf R},{\bf B}}(x+\lambda_1,y+\lambda_2)|^2=\sum_{\lambda_2\in\Lambda_2}\int_{X_1}|\widehat{\mu_x^2}(y+\lambda_2)|^2\,d\mu_1(x)$$$$
=\int_{X_1}\sum_{\lambda_2\in\Lambda_2}|\widehat{\mu_x^2}(y+\lambda_2)|^2\,d\mu_1(x)=\int_{X_1}1=1.$$
Thus $\Lambda_1\times\Lambda_2$ is a spectrum for $\mu_{\bf R,\bf B}$.

\medskip

For the converse, assume $\Lambda_1\times\Lambda_2$ is a spectrum for $\mu_{\bf R,\bf B}$. Take $\lambda_1\neq \lambda_1'$ in $\Lambda_1$ and $\lambda_2,\lambda_2'$ in $\Lambda_2$. Then
$$0=\widehat{\mu_{{\bf R}.{\bf B}}}(\lambda_1-\lambda_1',\lambda_2-\lambda_2')=\int_{X_1}e^{-2\pi i(\lambda_1-\lambda_1')\cdot x}\widehat{\mu_x^2}(\lambda_2-\lambda_2')\,d\mu_1(x).$$
But $\Lambda_1-\lambda_1'$ is a spectrum for $\mu_1$ so $\widehat{\mu_x^2}(\lambda_2-\lambda_2')=0$ for $\mu_1$-a.e. $x$. Since $\Lambda_2$ is countable, this implies that $\Lambda_2$ is an orthogonal set for $\mu_x^2$, for $\mu_1$-a.e. $x$. and in particular
\begin{equation}
\sum_{\lambda_2\in\Lambda_2}|\widehat{\mu_x^2}(y+\lambda_2)|^2\leq 1 \mbox{ for all $y\in\br^{d-r}$ and $\mu_1$-a.e. $x$}.
\label{eq1.3.1}
\end{equation}

Then

$$1=\sum_{\lambda_1\in\Lambda_1}\sum_{\lambda_2\in\Lambda_2}|\widehat\mu_{\bf R,\bf B}(x+\lambda_1,y+\lambda_2)|^2=\sum_{\lambda_2\in\Lambda_2}\sum_{\lambda_1\in\Lambda_1}\left|\int_{X_1}e^{-2\pi i \ip{ (\lambda_1+x)}{t}}\widehat{\mu_t^2}(y+\lambda_2)\,d\mu_1(t) \right|^2$$
$$=\sum_{\lambda_2\in\Lambda_2}\int_{X_1}|\widehat{\mu_t^2}(y+\lambda_2)|^2\,d\mu_1(t) \mbox{ (by the Parseval equality for the spectrum $\Lambda_1$)}$$
$$=\int_{X_1}\sum_{\lambda_2\in\Lambda_2}|\widehat{\mu_t^2}(y+\lambda_2)|^2\,d\mu_1(t) \leq \int_{X_1}1=1.$$
But combining with \eqref{eq1.3.1}, we get that, for a fixed $y$,

$$\sum_{\lambda_2\in\Lambda_2}|\widehat{\mu_t^2}(y+\lambda_2)|^2=1$$
for $\mu_1$-a.e. $t$. As $Q(y) : =\sum_{\lambda_2\in\Lambda_2}|\widehat{\mu_x^2}(y+\lambda_2)|^2$ is a continuous function, by Lemma \ref{JP} taking a countable dense set of $y$, we get that $\Lambda_2$ is a spectrum for $\mu_t^2$, for $\mu_1$-a.e. $t$ and this means $\Lambda_2$ is a spectrum, ${\mathbb P}$-almost surely.
\end{proof}

\medskip

In rest of the section, we will prove the spectral property result in Theorem \ref{th1.4}. To this end, we need to analyze a dynamical system generated by the Hadamard triple, for a detailed account of this dynamical system see \cite{DJ06,DJ07d}.

\begin{definition}\label{definv}
Let $(R,B,L)$ be a Hadamard triple. We define the function
$$m_B(x)=\frac{1}{\#B}\sum_{b\in B}e^{2\pi i \ip{b}{x}},\quad (\xi\in\br^d).$$
The Hadamard triple condition implies that $\delta_{R^{-1}B}$ is a spectral measure with a spectrum $L$ and $m_{R^{-1}B}$ is the Fourier transform of the Dirac measure.  Lemma \ref{JP} implies that
\begin{equation}\label{M_B}
\sum_{\ell\in L}|m_B((R^T)^{-1}(x+\ell))|^2=1, \mbox{or} \ \sum_{\ell\in L}|m_B(\tau_{\ell}(x))|^2=1,
\end{equation}
where we define the maps
$$\tau_{\ell}(x)=(R^T)^{-1}(x+\ell),\quad(x\in\br^d,\ell\in L), \ \mbox{and} \ \tau_{\ell_1...\ell_m} = \tau_{\ell_1}\circ...\circ\tau_{\ell_m}.
$$
A closed set $K$ in $\br^d$ is called {\it invariant (with respect to the system $(R,B,L)$)} if, for all $x\in K$ and all $\ell\in L$
$$
 m_B(\tau_{\ell}(x))>0 \ \Longrightarrow \ \tau_{\ell}(x)\in K.
 $$
We say that {\it the transition, using $\ell$, from $x$ to $\tau_\ell( x)$ is possible}, if $\ell\in  L$ and $m_B(\tau_{\ell} (x))>0$. A compact invariant set is called {\it minimal} if it does not contain any proper compact invariant subset.

For $\ell_1,\dots,\ell_m\in L$, the cycle $\mathcal C(\ell_1,\dots,\ell_m)$ is the set
$$\mathcal C(\ell_1,\dots,\ell_m)=\{x_0,\tau_{\ell_m}(x_0),\tau_{\ell_{m-1}\ell_m}(x_0),\dots, \tau_{\ell_2\dots\ell_m}(x_0)\},$$
where $x_0:=\wp(\ell_1,\dots, \ell_m)$ is the fixed point of the map $\tau_{\ell_1...\ell_m}$. i.e. $\tau_{\ell_1...{\ell_m}}(x_0)=x_0$.
The cycle $\mathcal C(\ell_1,\dots, \ell_m)$ is called an {\it extreme cycle for $(R,B,L)$} if $|m_B(x)|=1$ for all $x\in \mathcal C(\ell_1,\dots, \ell_m)$.
\end{definition}

\medskip

\begin{remark}\label{remark} Here are some remarks about the properties of extreme cycles.
\begin{enumerate}
\item For any extreme cycles, the only possible transition is from $x_0$ to $\tau_{\ell_m}(x_0)$  since $|m_B(\tau_{\ell_m}(x_0))|=1$ and (\ref{M_B}) implies all other must be zero.
\item Given any $c$ in an extreme cycle ${\mathcal C}$, we can always find another point in this cycle $c'$ such that $c = \tau_{\ell}(c')$ for a unique digit $\ell$ defining the cycle. Iterating the process, for any $n\geq 1$, $c = \tau_{\ell_0...\ell_{n-1}}(c')$ for some $c'\in{\mathcal C}$. Rewriting the relation, we have
    \begin{equation}\label{eqremark1}
    -c' = (R^T)^n(-c)+\ell_0 +R^T\ell_1+...+(R^T)^{n-1}\ell_{n-1}.
    \end{equation}
In particular, $(R^T)^nc$ is congruent modulo $\bz^d$ to another cycle point.
\item		
If $0\in B$, we have
\begin{equation}\label{eqremark2}
\langle R^n b,c\rangle\in\bz, \ \forall n\ge 0  \ \mbox{and} \ b\in B.
\end{equation}
First, $|m_B(c)|=1$ implies that $\langle b,c\rangle\in \bz$ for all $b\in B$ (we have equality in a triangle inequality so all the terms of the sum that defines $m_B$ must be equal to 1, since $0\in B$). In general, from (\ref{eqremark1}) and the fact that $c'$ is an extreme cycle point, $m_B(c')=1$ and $\ip{b}{c'}\in \bz$. As we know $\langle b,\ell\rangle\in \bz$, so we must have $\ip{R^nb}{c} = \ip{b}{(R^T)^nc}\in\bz$.
\end{enumerate}
\end{remark}

The following theorem shows the structure of  minimal compact invariant sets and we will use it throughout the rest of the paper.

\begin{theorem}\label{thccr}\cite[Theorem 2.8]{CCR}
Let $\M$ be a minimal compact invariant set contained in the zero set of an entire function $h$ on $\br^d$.
\begin{enumerate}
	\item There exists a proper rational subspace $V$ (can be $\{0\}$) invariant for $R^T$ such that $\M$ is contained in the union $\mathcal R$ of finitely many translates of $V$.
	\item This union contains the translates of $V$ by the elements of a cycle $\mathcal C(\ell_1,\dots, \ell_m)$ in $\M$, and $h$ is zero on $x+V$ for all $x\in\mathcal C(\ell_1,\dots,\ell_m)$.
	\item If the hypothesis ``(H) modulo $V$'' is satisfied, i.e., $(R^T)^{-1}(e_1-e_1')+(R^T)^{-2}(e_2-e_2')+\dots +(R^T)^{-p}(e_p-e_p')\in V$ implies $e_1-e_1',e_2-e_2',\dots, e_p-e_p'\in V$ for all $e_1,\dots,e_p,e_1',\dots, e_p'\in L$, then
	$$\mathcal R=\{x_0+V,\tau_{\ell_m}(x_0)+V,\dots,\tau_{\ell_2\dots \ell_m}(x_0)+V\}$$
	where $x_0=\wp(\ell_1,\dots \ell_m)$ and every possible transition from a point in $\M\cap (\tau_{\ell_q\dots \ell_m}(x_0)+V)$ leads to a point in $\M\cap(\tau_{\ell_{q-1}\dots\ell_m}(x_0)+V)$ for all $1\leq q\leq m$, with $\ell_0=\ell_m$.
	\item The union $\mathcal R$ is invariant.
\end{enumerate}
\end{theorem}

In particular, from (\ref{M_B}), extreme cycles are clearly compact invariant sets which correspond to the case $V=\{0\}$ (if needed, we can always take the entire function $h$ to be 0, in Theorem \ref{thccr}). However, the extreme cycles are not the only minimal compact invariant sets (see \cite{DJ07d} for some examples). We isolate this special case in the following definition.

\begin{definition}\label{definv2}
We say that the Hadamard triple $(R,B,L)$ is {\it dynamically simple} if the only minimal compact invariant set are extreme cycles. For a Hadamard triple $(R,B,L)$, the {\it orthonormal set $\Lambda$ generated by extreme cycles} is the smallest set such that
 \begin{enumerate}
 \item it contains $-{\mathcal C}$ for all extreme cycles ${\mathcal C}$ for $(R,B,L)$
 \item it satisfies $R^T\Lambda+L\subset \Lambda$.
 \end{enumerate}
When this set $\Lambda$ is a spectrum (see Theorem \ref{th_dynamic} below), we call it {\it the dynamically simple spectrum}. 

More generally, the {\it set generated by an invariant subset $A$} of $\br^d$, is the smallest set which contains $-A$ and satisfies (ii).
\end{definition}

\medskip

\begin{theorem}\label{th_dynamic}
Let  $(R,B,L)$ be a dynamically simple Hadamard triple. Then the orthonormal set $\Lambda$ generated by extreme cycles is a spectrum for the self-affine measure $\mu_{R,B}$ and $\Lambda$ is explicitly given by
$$\Lambda=\{\ell_0+R^T\ell_1+\dots (R^T)^{n-1}\ell_{n-1}+(R^T)^n(-c) : \ell_0,\dots,\ell_{n-1}\in L, n\geq0, \ c\mbox{ are extreme cycle points}\}.
$$
Moreover, if $(R,B,L)$ is a Hadamard triple on ${\mathbb R}^1$, it must be dynamically simple.
\end{theorem}

\begin{proof}
This theorem combines results in \cite{DJ06,DJ07d,DJ09}. An independent proof will be given in the appendix of this paper.
\end{proof}
\medskip

Our main theorem leading to main conclusion in the introduction is the following:

\begin{theorem}\label{th1.5}
Assume that the Hadamard triple $({\bf R},{\bf B},{\bf L})$ is in a quasi-product form defined in Definition \ref{def1.5} with $C=0$ and that $({\bf R},{\bf B},{\bf L})$ is dynamically simple. Let $\Lambda_1$  be the orthonormal set generated by extreme cycles for $(R_1,B_1,L_1)$ and suppose that $\Lambda_2$ is the set generated by those cycles which are extreme for \underline{all} triples $(R,B(i),L)$, $i=1,\dots, N_1$. Then  $ \Lambda_1 \times \Lambda_2$ is a spectrum for $({\bf R},{\bf B},{\bf L})$ and $\Lambda_2$ is a spectrum for $\mu_{\omega}$ for ${\mathbb P}$-almost every $\omega$.
\end{theorem}

\begin{proof}
As $({\bf R},{\bf B},{\bf L})$ is dynamically simple, we can define $\Lambda$ to be the dynamically simple spectrum for the quasi-product form $({\bf R},{\bf B},{\bf L})$.  We need to show that $\Lambda = \Lambda_1 \times \Lambda_2$. In the proof, it is worth to note that $|m_{B}(x)|=1$ if and only if $\ip{b}{x}\in{\mathbb Z}$ for all $b\in B$, since $0\in B$.

 \medskip


%

We show first that $\Lambda \subseteq \Lambda_1 \times \Lambda_2$. Property (ii) in Definition \ref{definv2} shows that the sets $\Lambda_1$ and $\Lambda_2$ satisfy $R_1^T \Lambda_1 +L_1 \subseteq \Lambda_1 $ and $R^T \Lambda_2 +L \subseteq \Lambda_2 $. With ${\bf R} = \left[
                               \begin{array}{cc}
                                 R_1 &0 \\
                                 0 & R \\
                               \end{array}
                             \right]
$, it is clear that
\beq
{\bf R}^T (\Lambda_1 \times \Lambda_2) + (L_1 \times L_2 ) \subseteq \Lambda_1 \times \Lambda_2.
\eeq
Thus we only have to show that $\Lambda_1 \times \Lambda_2$ contains $-C$ for all extreme cycles $C$ for $(\bf R,\bf B,\bf L)$ and then it follows from definition of $\Lambda$ that $\Lambda \subseteq \Lambda_1 \times \Lambda_2$.

\medskip

Let ${\mathcal C}=\{x_0 , x_1 , ..., x_{p-1} \}$ be such an extreme cycle of  $({\bf R},{\bf B},{\bf L})$. Then there exists ${\bf \ell} = (\ell_1, \ell_2) \in L_1 \times L_2 $ such that $x_{k+1} =( {\bf R}^T)^{-1} \left( x_k + (\ell_1 , \ell_2 )^T \right)$ for all $k=0,\dots,p-1$ and $x_p=x_0$. Writing $x_k= (x_k^{(1)},x_k^{(2)})$, we must have $x_{k+1}^{(1)} =( R_1^T)^{-1} \left( x_k^{(1)} + l_1  \right)$ and $x_{k+1}^{(2)} =( R^T)^{-1} \left( x_k^{(2)} +  l_2  \right)$ for all $k=0,\dots,p-1$. Thus the first components form a cycle for $(R_1, B_1, L_1)$ and the second components form a cycle for $(R_2^T, B(i), L)$ for all $i$. From the property of extreme cycle, we have that $\ip{b} {x_k} \in \bz$ for all $b \in {\bf B}$ and $k=0,1,...,p-1$. Therefore,
$$
\ip{a_i }{x_k^{(1)}} + \ip{d_{i,j}}{x_k^{(2)}} \in \bz
$$
for all $1 \leq i \leq N$, $1 \leq j \leq M$. Since  $d_{i,1} =0$, we must have $\ip{a_i }{ x_0^{(1)}} \in \bz$ for $1\leq i \leq N$ and therefore $\ip{d_{i,j}}{x_0^{(2)}} \in \bz$ for $1 \leq j \leq M$.
This shows that ${\mathcal C}_1 = \{ x_0^{(1)}, x_1^{(1)}, ... ,x_{p-1}^{(1)} \}$ is an extreme cycle for $(R_1 , B_1, L_1 )$, and  $C_2 = \{ x_0^{(2)}, x_1^{(2)}, ... ,x_{p-1}^{(2)} \}$ is an extreme cycle for all $(R_2^T, B(i), L_2)$.
Hence, $-C_1 \subseteq \Lambda_1$, $-C_2 \subseteq \Lambda_2$, and $-C \subseteq (-C_1) \times (-C_2) \subseteq \Lambda_1 \times \Lambda_2 $.
Since $\Lambda$ is the smallest set which is invariant under $R^T \Lambda + L$ and which contains $-C$ for all extreme cycles $C$, we must have $\Lambda \subseteq \Lambda_1 \times \Lambda_2 $.

\medskip

Next, we show that $\Lambda_1 \times \Lambda_2=\Lambda$. It suffices to show that $\Lambda_1\times\Lambda_2$ forms an orthogonal set for $\mu_{{\bf R},{\bf B}}$. Indeed, $\Lambda$ is a spectrum for $\mu_{{\bf R},{\bf B}}$ by Theorem \ref{th_dynamic}. This means that $\Lambda$ is a maximal orthogonal set (i.e. if $\lambda'\not\in\Lambda$, the exponential $e^{2\pi i \langle\lambda',x\rangle}$ cannot be orthogonal to all exponentials with frequencies in $\Lambda$). But $\Lambda\subset\Lambda_1\times\Lambda_2$ and $\Lambda_1\times\Lambda_2$ is a mutually orthogonal set, we must have $\Lambda = \Lambda_1\times\Lambda_2$.

\medskip

To show that $\Lambda_1\times\Lambda_2$ forms an orthogonal set for $\mu_{{\bf R},{\bf B}}$. We first note that as $\Lambda_1$ is a spectrum for the measure $\mu_1 = \mu(R_1,B_1)$ by Theorem \ref{th_dynamic}, $\Lambda_1$ is a mutually orthogonal set for   $\mu_1$. Hence,
\beq
\hat{\mu}_1 (\lambda_1 - \lambda_1' )  =0, \forall \lambda_1\neq\lambda_1'\in\Lambda_1.
\eeq
We now show that $\Lambda_2$ is a mutually orthogonal set for all $\mu_x^2$. Indeed, for all $x = x(i_1 , i_2 , ...)$,
\beq\label{three}
\hat{\mu}_x^{(2)}  (\lambda_2 - \lambda_2' ) = \prod_{k=1}^{\infty} m_{B_2 (i_k)} ( (R^T)^{-k} (\lambda_2 - \lambda_2') )
\eeq
where  $\lambda_2 \neq \lambda_2' \in \Lambda_2$. They can be written as
\beq\label{lambda_2}
\lambda_2 = \ell_0 + R^T \ell_1 + ... + (R^T)^{m-1} \ell_{m-1} + (R^T)^m (-x_0) ,
\eeq
\beq\label{lambda_2'}
\lambda_2' = \ell_0' + R^T \ell_1' + ... + (R^T)^{m'-1} \ell_{m'-1}' + (R^T)^{m'} (-x_0') ,
\eeq
with $\ell_i , \ell_i' \in L$, $x_0 , x_0'$ extreme cycle points for $(R, B(i), L)$. From (\ref{eqremark1}), for any $p\geq 1$, we can write
\beq \label{four}
-x_0 = (R^T)^k (-x_k) + \alpha_p + R^T \alpha_{p-1} + ... + (R^T)^{k-1} \alpha_{p-k} .
\eeq
Using (\ref{four}) in (\ref{lambda_2}), we can write $\lambda_2$ with as many digits as we want. Similarly, we can do this for case of $\lambda_2'$ in (\ref{lambda_2'}) and therefore we can take $m=m'$, and, as $\lambda_2,\lambda_2'$ are distinct elements,  we can assume that there exists $n< m$ such that $\ell_0 = \ell_0' , ... , \ell_{n-1} = \ell_{n-1}' , \ell_n \neq \ell_n'$. Then
$$
m_{B_2 (i_{n+1})} (  (R^T)^{-n-1} (\lambda_2 - \lambda_2' ) ) = m_{B_2 (i_n)} \left(  (R_2^T)^{-1} (\ell_n - \ell_n' ) + M_0+(R^T)^{m-n-1}(x-x')\right),
$$
where $M_0\in {\mathbb Z}^d$ and $x,x'$ are extreme cycle points. From integral periodicity of $m_{B_2 (i_n)} $ and (\ref{eqremark2}), the above quantity is equal to $ m_{B_2 (i_n)} (  (R_2^T)^{-1} (\ell_n - \ell_n' )) =0$ by the Hadamard triple assumption. This implies from (\ref{three}) that $\hat{\mu}_x^{(2)}  (\lambda_2 - \lambda_2' )=0$.

  \medskip

 If now $(\lambda_1 , \lambda_2) \neq (\lambda_1' , \lambda_2') \in \Lambda_1 \times \Lambda_2 $, we have, by Lemma \ref{lem1.23},
$$
\begin{aligned}
\langle e^{2\pi i\ip{( \lambda_1 , \lambda_2 )}{(x,y)} } , e^{2\pi i \ip{ ( \lambda_1' , \lambda_2' )}{(x,y)} } \rangle_{ L^2 (\mu_{{\bf R},{\bf B}} ) } =&  \int e^{2\pi i  \ip{(\lambda_1 - \lambda_1' , \lambda_2 - \lambda_2' ) }{ (x,y)}} d \mu_B (x,y) \\
=& \int \int  e^{ 2 \pi i \left( (\lambda_1 - \lambda_1')x + (\lambda_2 - \lambda_2' )y \right) } d \mu_x^{(2)} (y) d \mu_1 (x)  \\
= &\int e^{2 \pi i \ip{\lambda_1 - \lambda_1'}{x} } \widehat{\mu_x^{(2)}} (\lambda_2 - \lambda_2' ) d \mu_1 (x).
\end{aligned}
$$
As $\Lambda_2$ is a mutually orthogonal set for $\mu_x^{(2)}$, the term above is equal to $0$ if $\lambda_2 \neq \lambda_2'$. And if $\lambda_2 = \lambda_2'$, we must have $\lambda_1\neq\lambda_1'$ and hence $\widehat{\mu_1} (\lambda_1 - \lambda_1' ) =0$. Thus $\Lambda_1 \times \Lambda_2$ forms an orthogonal set and hence completes the proof that $\Lambda_1 \times \Lambda_2$ is a dynamically simple spectrum for $({\bf R},{\bf B},{\bf L})$.

\medskip

Finally, by Theorem \ref{th_dynamic}, $\Lambda_1\times\Lambda_2$ is a spectrum for the self-affine measure $\mu_{{\bf R},{\bf B}}$. Therefore, it follows from Theorem \ref{pr1.3} that $\Lambda_2$ is ${\mathbb P}$-almost surely a spectral measure for $\mu_{\omega}$.
\end{proof}

We now present the proof of Theorem \ref{th1.4}.



\begin{proof}[Proof of Theorem \ref{th1.4} (when (i) holds, i.e.,  the Hadamard triples $(R,B(i),L)$ are on ${\mathbb R}^1$)]
We pick a number $p\in\bn$ such that $pN\neq R$. Define the matrix
$${\bf R}=\begin{bmatrix}
	pN&0\\
	0&R
\end{bmatrix}.$$
Let $ \tilde B(i)=B(i (\mod \ N))$ for all $i\in\{0,1,...,pN-1\}$. Here, $i(\mod \ N)$ is the remainder when $i$ is divided by $N$. Let
$${\bf \tilde B}:=\left\{(i ,  d)^T : i\in\{0,1,\dots,pN-1\}, d\in \tilde B(i)\right\}.$$
Let ${\bf \tilde L}:=\{0,\dots,pN-1\}\times L$. In each of the coordinates, they form Hadamard triples on ${\mathbb R}^1$ and hence they must be dynamically simple by Theorem \ref{th_dynamic}.  We now show that $({\bf R},{\bf \tilde B},{\bf \tilde L})$ is also dynamically simple, so that Theorem \ref{th1.5} is applicable.

\medskip

Let ${\mathcal M}$ be a minimal compact invariant set. Assume that ${\mathcal M}$ is infinite and ${\mathcal M}$ is not an extreme cycle. Then , by Theorem \ref{thccr}, there is a subspace $V\neq \br^2$, invariant for $ {\bf R}$, such that $\M$ is contained in a union of finitely many translates of $V$. Since $V$ is invariant for ${\bf R}$, $pN\neq R$ and $V\neq\{0\}$ ($V=\{0\}$ corresponds to the extreme cycles), the only options are $V=\br\times\{0\}$ or $V=\{0\}\times \br$. We show that the first case is impossible while the second case implies all $\tilde B_j$ are the same, which means that Theorem \ref{th1.4} holds trivially.

\medskip

\noindent{\bf Case (i) $V=\br\times\{0\}$.} A direct check shows that the hypothesis ``(H) modulo $V$'' in Theorem \ref{thccr}(iii) is satisfied with ${\bf\tilde L}$ (See for example \cite[Proposition 3.7]{DJ07d} for an analogous proof). Applying now Theorem \ref{thccr}(iii), we deduce the existence of an ${\bf \tilde L}$-cycle  $(x_0,y_0)$, with digits $(i_1,\ell_1),\dots,(i_m,\ell_m)$ ($\ell_j\in L$, $i_j\in\{0,1,...,pN-1\}$) such that
$$\M\subset\bigcup_{k=1}^m\left(\tau_{(i_k,\ell_k)}\dots \tau_{(i_m,\ell_m)}(x_0,y_0)+V\right)=:\mathcal R$$
and $\mathcal R$ is invariant. Moreover, every possible transition from $\tau_{(i_k,\ell_k)\dots (i_m,\ell_m)}(x_0,y_0)+V$ leads to a point in $\tau_{(i_{k-1},\ell_{k-1})\dots(i_m,\ell_m)}(x_0,y_0)+V$ for $1\leq k\leq m$, where $(i_0,\ell_0):=(i_m,\ell_m)$.

\medskip

Let $(x,y_0)\in (x_0,y_0)+V$. Let $\ell\neq \ell_m$. Then $\tau_{(i',\ell)}(x,y_0)\not \in \tau_{(i_m,\ell_m)}(x_0,y_0)+V$, thus the transition is not possible so $m_{\bf\tilde B}(\tau_{(i',\ell)}(x,y_0))=0$ which means
$$\sum_{k=0}^{pN-1}\sum_{d\in \tilde B(k)}e^{2\pi i (k\frac{x+i'}{pN}+d\frac{y_0+\ell}R)}=0.$$
Let $x':=\frac{x+i'}{pN}$. Then
$$0=\sum_{k=0}^{pN-1}e^{2\pi i kx'}\sum_{d\in\tilde B_2(k)}e^{2\pi i d\frac{y_0+l}{R}}.$$
Since $x'$ can be any real number, the coefficients of this polynomial must be zero:
$$\sum_{d\in \tilde B(k)}e^{2\pi i  d\frac{y_0+\ell}{R}}=0,\quad \mbox{ for all }\ell\neq \ell_m$$
by the linear independence of the trigonometric polynomials $e^{2\pi i kx'}$, $k=0,1,...,pN-1$. This means that $m_{\tilde B(k)}(\tau_{\ell}(y_0))=0$ for all $\ell\neq \ell_m$ and hence $|m_{\tilde B(k)}(\tau_{\ell_m}(y_0))|=1$. Since $0\in \tilde B(k)$ we have equality in the triangle inequality, so $m_{\tilde B(k)}(\tau_{\ell_m}(y_0))=1$. We can do this for all the points in the cycle $C_2:=\{y_k:=\tau_{\ell_k}\dots \tau_{\ell_m}(y_0) : 1\leq k\leq m\}$, and we conclude that $C_2$ is an extreme $L$-cycle  for all $\tilde B(k)$.

\medskip

We now consider $v:=(x,y_k)\in \M$ and the possible transition  from $v$, which must be of the form $(i',\ell_{k-1})$. From the extreme cycle property, we have $m_{\tilde B(k)}(y_{k-1})=1$ and hence
$$0\neq m_{\bf\tilde B}(\tau_{(i',\ell_{k-1})}(x,y_k))=\frac{1}{pN}\sum_{j=0}^{pN-1}e^{2\pi i j\frac{x+i'}{pN}}m_{\tilde B(k)}(y_{k-1})=\sum_{j=0}^{pN-1}e^{2\pi i j\frac{x+i'}{pN}}.$$
This holds if and only if $x+i'\not\in\bz$ or $x+i'$ is a multiple of $pN$. As $\M$ is infinite and $\M\subset T({\bf R},{\bf\tilde L}) =[0,1]\times T(R,L)$, we can assume that $x\not\in \bz$. In this case, when $i'=0$,
$$m_{\bf\tilde B}(\tau_{(0,\ell_{k-1})}(x,y_{k}))=\sum_{j=0}^{pN-1}e^{2\pi i j\frac{x}{pN}}\neq 0.$$
Hence the transition is possible and we conclude that $({x}/{pN},y_{k-1})$ is in $\M$. Iterate this step by replacing $(x,y_{k})$ with $({x}/{pN},y_{k-1})$.  Taking the limit and using the compactness of $\M$, we obtain that $\M$ contains $(0,y_k)$ for all $k$. But that means that $\M$ contains an  ${\bf \tilde L}$-cycle which is extreme for ${\bf \tilde B}$, and by minimality, it has to be equal to the extreme cycle. That is a contradiction. Thus, $V$ cannot be $\br\times \{0\}$.

\medskip

\noindent{\bf Case (ii) $V=\{0\}\times\br$.} As before, ``(H) modulo $V$'' in Theorem \ref{thccr}(iii) is satisfied and Theorem \ref{thccr} implies that there exists $(x_0,y_0)$, and an $\tilde L$-cycle, with digits $(i_1,\ell_1),\dots,(i_m,\ell_m)$ such that
$$M\subset\bigcup_{k=1}^m\left(\tau_{(i_k,\ell_k)}\dots \tau_{(i_m,\ell_m)}(x_0,y_0)+V\right)=:\mathcal R,$$
$\mathcal R$ is invariant and every possible transition from $\tau_{(i_k,\ell_k)}\dots \tau_{(i_m,\ell_m)}(x_0,y_0)+V$ leads to a point in $\tau_{(i_{k-1},\ell_{k-1})}\dots \tau_{(i_m,\ell_m)}(x_0,y_0)+V$ for $1\leq k\leq m$, where $(i_0,\ell_0):=(i_m,\ell_m)$.

\medskip

Take $i'\neq i_m$ in $\{0,\dots,pN-1\}$ and $y\in{\mathbb R}$. The transition  from $(x_0,y)$ to $ \tau_{(i',\ell)}(x_0,y)$ is not possible and thus $m_B(\tau_{(i,\ell)}(x_0,y))=0$.  Then with $y'=(y+\ell)/R$,
$$
0=\sum_{j=0}^{pN-1}\sum_{d\in \tilde B(j)}e^{2\pi i j\frac{x_0+i'}{pN}}e^{2\pi i d\frac{y+\ell}{R}} =\sum_{d\in\cup_j \tilde B(j)}e^{2\pi i dy'}\sum_{\{j: d\in \tilde B(j)\}}e^{2\pi ij\frac{x_0+i}{pN}}.$$
Then, all the coefficients are zero so, for all $d\in \cup_{j=1}^{N} \tilde B(j)$, and all $i'\neq i_m$,
$$
\sum_{j : d\in \tilde B(j)}e^{2\pi i j\frac{x_0+i}{pN}}=0.
$$
But $0\in \tilde B(j)$ for all $j$ so
$$\sum_{j=0}^{pN-1}e^{2\pi i j\frac{x_0+i'}{pN}}=0$$
for all $i'\neq i_m$. As the same time, this implies that $\sum_{j=0}^{pN-1}e^{2\pi i j\frac{x_0+i_m}{pN}}=1$ and hence $x_0\equiv (-i_m)(\mod pN)$. Since $x_0\in[0,1]$ (by  $\M\subset T({\bf R},{\bf\tilde L}) =[0,1]\times T(R,L)$) we obtain that $x_0=0$ or $x_0=1$.

\medskip

If $x_0=0$, then the digits  corresponding to this cycle are $m=1$ and $i_1=0$. Then we have
$$\sum_{j : d\in \tilde B_2(j)}e^{2\pi i j\frac{0+i'}{pN}}=0$$
for all $i'\neq i_1=0$ and all $d\in \cup_{j=1}^{N}\tilde B(j)$. Let $A_d(x):=\sum_{j: d\in \tilde B_2(j)}x^j$. Then $A_d(e^{2\pi i \frac{i'}{pN}})=0$ for all $i'\in\{1,\dots, pN-1\}$. Therefore $A_d$ is divisible by $1+x+\dots +x^{pN-1}$ and this implies that $A_d(x)=1+x+\dots +x^{pN-1}$. So every $d\in \cup_{j=1}^{N} \tilde B(j)$ appears in all $\tilde B(j)$. But this means that all the sets $\tilde B(j)$ are equal  and so all $\mu_\omega=\mu_{R,\tilde B(0)}$ which is the self-affine spectral measure. The conclusion holds trivially. Similarly, the case $x_0=1$ follows from the same argument , the cycle has  digits $m=1$ and $i_1=pN-1$.

\medskip

Now, we can see that the only minimal compact invariant sets are extreme cycles. By Theorem \ref{th1.5}, with $\Lambda_2$ as defined in its hypothesis, we have that $\mu_x^2$ has spectrum $\Lambda_2$ for $\mu_1$-a.e. $x$. Note that $\mu_1$ is the Lebesgue measure. Then $\mu_\omega$ has spectrum $\Lambda_2$ for $\tilde {\mathbb P}$-a.e. $\omega\in\{0,\dots,pN-1\}^{\bn}$, where $\tilde{\mathbb P}$ is the product probability measure on $\{0,1,\dots,pN-1\}^{\bn}$ that assigns equal probabilities $\frac{1}{pN}$ to every digit $0,1,\dots,pN-1$.
Consider now the map
$$\Phi:\{0,1,\dots, pN-1\}^{\bn}\rightarrow\{0,1,\dots,N-1\}^{\bn},\quad \Phi(i_1i_2\dots)=(i_1(\mod N), i_2(\mod N), \dots).$$
 By checking on cylinder sets, note that for any Borel subset of $\{0,1,\dots,N-1\}^{\bn}$,
$$\mathbb P(E)=\tilde {\mathbb P}(\Phi^{-1}(E)).$$
Also, note that for $\omega=i_1i_2\dots\in\{0,1,\dots,pN-1\}^{\bn}$, we have $\mu_{\omega}=\mu_{\Phi(\omega)}$, because $\tilde B_2(i)=B(i(\mod N))$. Then
$$\mathbb P(\omega :\mu_\omega\mbox{ has spectrum $\Lambda_2$})=\tilde{\mathbb P}(\omega :\mu_{\Phi(\omega)}\mbox{ has spectrum $\Lambda_2$})=\tilde{\mathbb P}(\omega : \mu_\omega\mbox{ has spectrum }\Lambda_2)=1.$$
This completes the proof.
\end{proof}

\medskip

\begin{proof}[Proof of Theorem \ref{th1.4} (when (ii) holds, i.e.,  each $B(i)$ is a complete set representative of $R$)] Consider
$${\bf R}=\begin{bmatrix}
	N&0\\
	0&R
\end{bmatrix}, \mbox{and}$$
$${\bf \tilde B}:=\left\{(i ,  d)^T : i\in\{0,1,\dots,N-1\}, d\in  B(i)\right\}.$$
 In this case, the attractor $T({\bf R},{\bf \tilde B})$ defined in (\ref{T(R,B)}) is a self-affine tile and it admits a lattice tiling of the form ${\bz}\times \tilde\Gamma$ for some lattice $\tilde\Gamma$ (See e.g. \cite{LW2} and \cite[Proposition 4.4, Claim]{DHL15}). Hence, it admits a spectrum of the form ${\bz}\times \Gamma$ with $\Gamma$ a dual lattice of $\tilde\Gamma$. Hence, Theorem \ref{pr1.3} shows that $\Gamma$ is almost surely a spectrum for $\mu_\omega$.

The final statement in Theorem \ref{th1.4} will be proved via the following general lemma. \end{proof}

\medskip

\begin{lemma}\label{lem3.11}
Let $\mu$ be a Borel, compactly supported probability measure on $\br^d$. Suppose $\mu$ is spectral and the spectrum is a full-rank lattice $\Gamma$. Then $\mu$ is the Lebesgue measure with support $T$ which tiles $\br^d$ by the dual lattice $\tilde\Gamma$.
\end{lemma}
\begin{proof}
Let $\Gamma=A\bz^d$ for some integer $d\times d$ non-singular matrix $A$. The dual lattice is $\tilde\Gamma=(A^T)^{-1}\bz^d$. We change the variable to reduce the problem to the case when $\Gamma=\bz^d$. Define the Borel probability measure $\nu$ by
$$\int f(x)\,d\nu(x)=\int f(A^Tx)\,d\mu(x),$$
for all continuous functions $f$ on $\br^d$.
Then $\nu$ has spectrum $\bz^d$. By Lemma \ref{JP}, $\sum_{n\in{\bz^d}}|\widehat{\nu}(\xi+n)|^2=1$. Thus,
$$
\int_{{\mathbb R}^d}|\widehat{\nu}(\xi)|^2d\xi = \int_{[0,1)^d}\sum_{n\in{\bz^d}}|\widehat{\nu}(\xi+n)|^2d\xi =1.
$$
 This shows that $\nu$ is absolutely continuous with respect to the Lebesgue measure. As $\nu$ is a spectral measure, we must have $\nu = \frac{1}{\mbox{Leb(S)}}\chi_Sdx$ for some measurable set $S$ (Theorem 1.5 in \cite{DL14}). $S$ is therefore a spectral set with spectrum $\bz^d$. By the well-known theorem of Fuglede \cite{Fug74}, $S$ is a translational tile by tiling set $\bz^d$. 
 This implies that $\mu$ is the normalized Lebesgue measure on the set $T:=(A^T)^{-1}S$, which tiles $\br^d$ by $(A^T)^{-1}\bz^d=\tilde\Gamma$.

See also \cite[Theorem 2.4]{DuJo13} for a variation of the proof. 
\end{proof}

\medskip

\section{Appendix: dynamically simple spectrum}

We will prove Theorem \ref{th_dynamic} in this section. We let $\Lambda$ be the orthonormal set generated by the extreme cycles for $(R,B,L)$ and $$
\Lambda'=\{\ell_0+R^T\ell_1+\dots (R^T)^{n-1}\ell_{n-1}+(R^T)^n(-c) : \ell_0,\dots,\ell_{n-1}\in L, n\geq0, \ c\mbox{ are extreme cycle points}\}
$$
the set given in Theorem \ref{th_dynamic}. We first prove from definition that they are the same.

\medskip

\begin{lemma}\label{lem4.1}
$\Lambda=\Lambda'$. In fact,
$$
\Lambda= R^T\Lambda+L.
$$
\end{lemma}

\begin{proof}
It is clear that $R^{T}\Lambda'+L \subset \Lambda'$. Also, for any extreme cycle points $c$, there exists unique $\ell$ such that  $c' = \tau_{\ell}(c)$ is an extreme cycle point. Hence, $-c= \ell+R^T(-c')$, which implies $\Lambda'$ contains all extreme cycles. By definition, $\Lambda\subset \Lambda'$. On the other hand, since $-c\in \Lambda$, the invariance implies that $\ell_m+R^{T}(-c)\in \Lambda$ for all $\ell_n\in L$. Inductively, $(R^{T})^2(-c)+R^T\ell_n+\ell_{n-1}\in \Lambda$, and in the end, $$
\ell_0+R^T\ell_1+\dots (R^T)^{n-1}\ell_{n-1}+(R^T)^n(-c)\in \Lambda,$$
for all $n$. Thus $\Lambda'\subset \Lambda$. This shows $\Lambda = \Lambda'$. From the definition of $\Lambda'$, it is clear that $\Lambda = R^T\Lambda+L$.
\end{proof}

\medskip

From now on, we will work on the expression $\Lambda'$, and for simplicity, we still write it as $\Lambda$. We first show the mutually orthogonality of $\Lambda$ in $\mu_{R,B}$.

\begin{proposition}\label{pr4.1}
$\Lambda$ is a mutually orthogonal set in $\mu_{R,B}$.
\end{proposition}

\begin{proof}
We need to see whether
\beq\label{three1}
\hat{\mu} (\lambda - \lambda' ) = \prod_{k=1}^{\infty} m_{B} ( (R^T)^{-k} (\lambda - \lambda') )
\eeq
is zero whenever  $\lambda \neq \lambda' \in \Lambda$. Now, they can be written as
\beq\label{lambda_21}
\lambda = \ell_0 + R^T \ell_1 + ... + (R^T)^{m-1} \ell_{m-1} + (R^T)^m (-c) ,
\eeq
\beq\label{lambda_21'}
\lambda' = \ell_0' + R^T \ell_1' + ... + (R^T)^{m'-1} \ell_{m'-1}' + (R^T)^{m'} (-c') ,
\eeq
with $\ell_i , \ell_i' \in L$, $c , c'$ extreme cycle points for $(R, B, L)$ (they may be from different cycles). From (\ref{eqremark1}), for any $p\geq 1$, we can write
\beq \label{four1}
-c = (R^T)^k (-c_k) + \alpha_p + R^T \alpha_{p-1} + ... + (R^T)^{k-1} \alpha_{p-k}
\eeq
for some digits $\alpha_i$ in $L$ and another extreme cycle point $c_k$.
 Using  (\ref{four1}) in (\ref{lambda_21}), we can write $\lambda$ with as many digits as we want. Similarly, we can do it for case of $\lambda'$ in (\ref{lambda_21'}). As $\lambda,\lambda'$ are distinct elements,  we can assume for some $m=m'$ that there exists $n< m$ such that $\ell_0 = \ell_0' , ... , \ell_{n-1} = \ell_{n-1}' , \ell_n \neq \ell_n'$.
$$
m_{B} (  (R^T)^{-n-1} (\lambda - \lambda' ) ) = m_{B} \left(  (R_2^T)^{-1} (\ell_n - \ell_n' ) + M_0+(R^T)^{m-n-1}(x-x')\right),
$$
where $M_0$ is some integer vector in ${\mathbb Z}^d$ and $x,x'$ are extreme cycle points. From the integral periodicity of $m_{B} $ and (\ref{eqremark2}), $\ip{b}{M}\in\bz$ and $\ip{b}{(R^T)^{m-n-1}(x-x')}\in\bz$,  The term above is equal to $ m_{B} (  (R_2^T)^{-1} (\ell_n - \ell_n' )) =0$ by the Hadamard triple assumption. This implies from (\ref{three1}) that $\hat{\mu}  (\lambda - \lambda' )=0$.
\end{proof}

\medskip

We need an easy geometric lemma.

\begin{lemma}\label{lem4.2}
Suppose that $(R,B,L)$ is a Hadamard triple and we define $\tau_{\ell}(x) = (R^T)^{-1}(x+\ell)$. Let ${\mathcal B}_r$ be the closed Euclidean ball centered at origin. Then for $r$ sufficiently large,
$$
\bigcup_{\ell\in L}\tau_{\ell}({\mathcal B}_r)\subset {\mathcal B}_r.
$$
\end{lemma}

\begin{proof}
Denote by $|\cdot|$ the Euclidean distance on ${\mathbb R}^d$. 
Since $R$ is expansive, there exists $0<c>1$ such that $|(R^T)^{-1}v|\leq c|v|$ for all $v\in \br^d$. 
Let
$$
M = \max\{|(R^T)^{-1}\ell|: \ell\in L\}
$$
and let $r>cM/(1-c)$. 
 Then for all $\ell\in L$ and $|x|\leq r$, we have
$$
|\tau_{\ell}(x)|\leq c(r+M) <r.
$$
Hence, $\tau_{\ell}({\mathcal B}_r)\subset {\mathcal B}_r$. This proves the lemma.
\end{proof}
\medskip

\begin{proof}[Proof of Theorem \ref{th_dynamic}] As mutually orthogonality has been established in Proposition \ref{pr4.1}, we just need to show that the set $\Lambda$ generated by the extreme cycle is complete. By Jorgensen-Pedersen Lemma (Lemma \ref{JP}), we need to show
$$
Q_{\Lambda}(\xi) : = \sum_{\lambda\in\Lambda}|\widehat{\mu}(\xi+\lambda)|^2 =1, \ \forall \xi\in{\mathbb R}^d.
$$
In fact, $Q_{\Lambda}\leq 1$ is well-known by mutually orthogonality. We just need to see whether $Q_{\Lambda}\geq 1$. To do this, we define the {\it Ruelle transfer operator}
$$
{\mathcal R} f(\xi): = \sum_{\ell\in L}|m_B(\tau_{\ell}(\xi))|^2f(\tau_{\ell}(\xi)).
$$
Using Lemma \ref{lem4.2}, we choose $r$ large enough such that the closed ball ${\mathcal B}_r$ satisfies
\begin{equation}\label{B_R}
\bigcup_{\ell\in L}\tau_{\ell} ({\mathcal B}_r)\subset{\mathcal B}_r
\end{equation}
and let $c_r = \min_{\xi\in {\mathcal B}_r} Q_{\Lambda}(\xi)$. Then ${\mathcal R} c_r = c_r$. On the other hand, as $\Lambda$ satisfies $R^T\Lambda+L = \Lambda$ by Lemma \ref{lem4.1}, we have
$$
\begin{aligned}
Q_\Lambda(\xi) =& \sum_{\ell\in L}\sum_{\lambda\in\Lambda}|\widehat{\mu}(\xi+R^T\lambda+\ell)|^2\\
=& \sum_{\ell\in L}\sum_{\lambda\in\Lambda}|m_B((R^T)^{-1}(\xi+\ell))|^2|\widehat{\mu}(({R^T})^{-1}(\xi+\ell)+\lambda)|^2 = \left({\mathcal R} Q_{\Lambda}\right)(\xi).\\
\end{aligned}
$$
Thus,  if we define
$$
f_n = Q_{\Lambda}-c_r,
$$
then ${\mathcal R}f_n = f_n$ and $f_n$ is an entire function. Consider the set in ${\mathcal B}_r$ for which $Q_n$ attains minimum, i.e.,
$$
\M_0 = \{\xi\in {\mathcal B}_r: f_n(\xi)=0\}.
$$
We note that $\M_0$ is a compact invariant set in ${\mathcal B}_r$. To show the invariance, we suppose $\xi\in \M$ and $|m_B(\tau_{\ell}(\xi))|>0$. As
$$
0=f_n (\xi) = \sum_{\ell\in L} |m_B(\tau_{\ell}(\xi))|^2 f_n(\tau_{\ell}(\xi))
$$
and $f\geq 0$, $|m_B(\tau_{\ell}(\xi))|^2 f_n(\tau_{\ell}(\xi))=0$ and hence $f_n(\tau_{\ell}(\xi))=0$. Because of (\ref{B_R}), $\tau_{\ell}(\xi)\in {\mathcal B}_r$. Take a minimal compact invariant set $\M\subset\M_0\subset {\mathcal B}_R$. The crux of the proof is to note that {\it the dynamically simple Hadamard triple assumption forces $\M$ to be an extreme cycle}. But extreme cycles are contained in $\Lambda$, this in turn shows that there are some points (indeed the whole extreme cycle) $x_0\in \Lambda\cap \M_0$. By mutual orthogonality, $Q_{\Lambda}(x_0)=1$,
$$
f_n(x_0) =0, \ \mbox{and} \ c_r = Q_{\Lambda}(x_0) = 1.
$$
Hence, $ \min_{\xi\in {\mathcal B}_r} Q_{\Lambda}(\xi) =1$. But $r$ can be arbitrarily large this shows $Q_{\Lambda}(\xi) \ge1$ for $\xi\in\br^d$.

\medskip

Finally, on ${\mathbb R}^1$, the zero set of an entire function must be a discrete set, showing that any minimal invariant set contained in $\mathcal M_0$ must be discrete, which must be an extreme cycle as the subspace can only be $V = \{0\}$ by Theorem \ref{thccr}.
\end{proof}

The idea of cycle points is also related to the integer points inside a self-affine fractal. A study in this direction can be found in \cite{GY}.

 \begin{acknowledgements}
This work was partially supported by a grant from the Simons Foundation (\#228539 to Dorin Dutkay) and Chun-Kit Lai was supported by the mini-grant by ORSP of San Francisco State University (Grant No: ST659). The authors would like thank Dr. John Hausserman and Jean-Pierre Gabardo for many insightful discussions.
\end{acknowledgements}

\bibliographystyle{alpha}	
\bibliography{eframes}

\end{document}